\documentclass[a4paper,12pt,leqno]{article}
\usepackage{latexsym}

\usepackage{amssymb} 
\usepackage{amsmath} 
\usepackage{amsthm}

\usepackage{amscd}
\usepackage{bm}
\usepackage{enumerate}
\usepackage[dvips]{graphicx,color}

\def\R{{\mathbb{R}}}
\def\A{{\mathcal{A}}}
\def\Shi{{\mathcal{S}}}

\numberwithin{equation}{section}

\newtheorem{theorem}{Theorem}[section]
\newtheorem{prop}[theorem]{Proposition}

\newtheorem{lemma}[theorem]{Lemma}
\newtheorem{define}[theorem]{Definition}
\newtheorem{rem}[theorem]{Remark}

\usepackage{amsthm}

\title{
A basis construction for the Shi arrangement of the type $B_{\ell}$ or $C_{\ell}$}
\author{Daisuke Suyama
\footnote
{
Research Fellow of the Japan Society for the Promotion
of Science.
}
\\
\\
{\footnotesize {\it Department of Mathematics, Hokkaido University, 
Sapporo, Hokkaido 060-0810, Japan.}}
\\
{\footnotesize (email: dsuyama@math.sci.hokudai.ac.jp)}
}

\date{}

\begin{document}

\maketitle

\begin{abstract} 
The Shi arrangement is an affine 
arrangement of hyperplanes 
consisting of the hyperplanes 
of the Weyl arrangement
and their parallel translations.
It was introduced by J.-Y. Shi 
in the study of the Kazhdan-Lusztig representation of the affine Weyl groups.
M. Yoshinaga showed that the cone over 
every Shi arrangement
is free.
In this paper, 
we construct an explicit basis
for the derivation module
of the cone over the Shi arrangements 
of the type $B_{\ell}$ or $C_{\ell}$.
\end{abstract} 

{\footnotesize {\it Keywords}: Hyperplane arrangement; Shi arrangement; Free arrangement; Derivations}

\section{Introduction} 

Let $E$ be an $\ell$-dimensional real Euclidean space.
Let $\Phi$ be an irreducible root system and
$\Phi_{+}$ denote the set of positive roots of $\Phi$. 
The 
{\bf Weyl arrangement} of the type $\Phi$ is denoted by
$\A(\Phi)$:
$$
\mathcal{A}(\Phi)=\{ H_{\alpha} \mid \alpha \in \Phi_{+} \}, \,\,
{\rm where~} H_{\alpha}=\{ v \in E \mid \alpha(v)=0 \}.
$$ 
Let
\[
H_{\alpha,1} = \{ v \in E \mid \alpha(v)=1 \}.
\]
Then the {\bf Shi arrangement} is given by
\[
\mathcal{A}(\Phi) \cup \{ H_{\alpha,1} \mid \alpha \in \Phi_{+} \} = \bigcup_{\alpha \in \Phi_{+}} \{ H_{\alpha},
H_{\alpha,1} \}.
\]
Embed the $\ell$-dimensional space $E$ into $V=\R^{\ell+1} 
$ by 
adding a new coordinate $z$ such that $E$ is defined by 
the equation $z=1$ in $V$. 
Then, as in \cite[Definition 1.15]{OT},
we have the {\bf cone}  $\Shi (\Phi) $ over the Shi arrangement.
It is a central arrangement in $V$
defined by
\[
Q(\Shi (\Phi) ) = 
z \prod_{\alpha \in \Phi_{+}} \alpha (\alpha -z) =0.
\]
Let $S$ be the algebra of polynomial functions on $V$ and let $\mathrm{Der}(S)$ be the module of derivations 
of $S$ to itself
\[
\mathrm{Der}(S)=\{ \theta : S \rightarrow S \mid \theta\ \text{is}\ \mathbb{R} \text{-linear and}\ \theta (fg) =f\theta(g)+ g\theta(f)\ \text{for any}\ f,g \in S \}.
\]
The derivation module $D(\Shi (\Phi))$ is defined by 
\begin{multline*} 
D(\Shi (\Phi))=\{ \theta \in \mathrm{Der}(S) \mid
\theta(z)\ \text{is divisible by}\ z,  
\theta(\alpha)\ \text{is divisible by}\ \alpha\\
\text{ and }\ 
\theta(\alpha-z)\ \text{is divisible by}\ \alpha-z
\text{ for any}\ \alpha \in \Phi_{+} \}.
\end{multline*} 
We say that $\Shi (\Phi)$ is {\bf free} if $D(\Shi (\Phi))$ is a free $S$-module. 

Let $x_{1},\ldots ,x_{\ell}$ be an orthonormal basis for the dual space $E^{*}$.
In this paper we explicitly choose root systems $\Phi^{B}$ and $\Phi^{C}$, 
and positive root systems $\Phi_{+}^{B}$ and $\Phi_{+}^{C}$
of the types $B_{\ell}$ and $C_{\ell}$ respectively as follows:
\begin{align*} 
\Phi^{B} &:= 
\{ \pm x_{i}, \pm x_{p} \pm x_{q} \in E^{*}
\mid 1\leq i\leq \ell,1\leq p < q\leq \ell \},
\\
\Phi_{+}^{B} &:=
\{ x_{i}, x_{p} \pm x_{q} \in \Phi^{B} \mid 1\leq i\leq \ell , 1\leq p < q\leq \ell \},
\\
\Phi^{C} &:= 
\{ \pm 2x_{i}, \pm x_{p} \pm x_{q} \in E^{*}
\mid 1\leq i\leq \ell,1\leq p < q\leq \ell \},
\\
\Phi_{+}^{C} &:=
\{ 2x_{i}, x_{p} \pm x_{q} \in \Phi^{C} \mid 1\leq i\leq \ell,1\leq p < q\leq \ell \}.
\end{align*} 
We express the cones over the Shi arrangements of 
the types $B_{\ell}$ and $C_{\ell}$ by
$\Shi (B_{\ell})$ and $\Shi (C_{\ell})$ respectively.  

In the study of the Kazhdan-Lusztig
representation theory of the affine Weyl groups,
J.-Y. Shi introduced 
the Shi arrangements for the type $A_{\ell}$ 
in \cite{shi1}.
Later a good number of articles, including \cite{Ath,ER,Hea,St96,Y04},
study the Shi arrangements. 
M. Yoshinaga proved in \cite{Y04}
that the cone over the Shi arrangement is a free arrangement
 by settling the Edelman-Reiner 
conjecture in \cite{ER} which asserts that the generalized 
Shi and Catalan arrangements are free.
However, even in the case of the cone over the Shi arrangement 
of the type $A_{\ell}$, 
no basis was constructed explicitly at that time.
Recently a basis for the cone over the Shi arrangement of the type $A_{\ell}$
is constructed explicitly in \cite{Sute}
and of the type $D_{\ell}$ in \cite{Gao}.
In those papers the most important ingredients of their 
recipes are 
the Bernoulli polynomials $B_{k}(x)$
and their relatives $B_{r, s}(x)$.
In the present paper, we construct bases for
the cones over the Shi arrangements of the types $B_{\ell}$ and $C_{\ell}$
by using new Bernoulli-like polynomials $B_{r,s}^{B}(x)$ and $B_{r,s}^{C}(x)$.

The organization of this paper is as follows: 
in Section 2, we will construct $\ell$ derivations 
$
\varphi_{1}^{B},
\dots,
\varphi_{\ell}^{B}
$ 
belonging to $D(\Shi(B_{\ell}))$.
In Section 3, we will prove that they,
together with the Euler derivation,
 form a basis of 
$D(\Shi(B_{\ell}))$.
In Section 4, we 
present a similar construction of a basis for 
$D(\Shi(C_{\ell}))$ for the type $C_{\ell}$.

\section{A basis construction for the type $B_{\ell}$}

\begin{define}
\label{defBofB} 
For $(r,s) \in \mathbb{Z}_{> 0}\times \mathbb{Z}_{\geq 0}$, 
define a polynomial
$B_{r,s}^{B} (x)$ in $x$ satisfying the following two conditions:
\begin{enumerate}[(i)]
\item $B_{r,s}^{B} (x+1)-B_{r,s}^{B} (x)=\displaystyle\frac{(x+1)^{r}-(-x)^{r}}{(x+1)-(-x)} (x+1)^{s}(-x)^{s}$,
\item $B_{r,s}^{B} (0)=0$.
\end{enumerate}
\end{define}

Note that $\frac{(x+1)^{r}-(-x)^{r}}{(x+1)-(-x)}$
is a polynomial either 
of degree $r-1$ (when $r$ is odd)
or 
of degree $r-2$ (when $r$ is even).
It is thus
easy to see that $B_{r,s}^{B} (x)$
uniquely exists and
\[
\deg B_{r,s}^{B} (x)
=
\begin{cases}
r+2s   &{\rm ~if~} r {\rm ~is~odd},\\
r+2s-1 &{\rm ~if~} r {\rm ~is~even}.
\end{cases}  
\]

\begin{lemma}
\label{oddfunctionB}
$B_{r,s}^{B} (x)$ is an odd function.
%
\end{lemma}

\begin{proof}
Replacing $x$ with $-x-1$ in \ref{defBofB} (i), we have
\begin{align*}
B_{r,s}^{B}(-x)-B_{r,s}^{B}(-x-1) &= \frac{(-x)^{r}-(x+1)^{r}}{(-x)-(x+1)} (-x)^{s}(x+1)^{s} \\
&= B_{r,s}^{B}(x+1)-B_{r,s}^{B}(x).
\end{align*}
Then we get
$
F(x) = F(x+1)
$
where
$F(x):=B_{r,s}^{B}(x)+B_{r,s}^{B}(-x).$
Thus we obtain
\[ 
F(n)
=F(n-1)
=\dots
=F(0)=0
\,\,\,
(n \in \mathbb{Z}_{\geq 0})
\]
and
\[ B_{r,s}^{B}(x)+B_{r,s}^{B}(-x)=F(x)=0. \]
\end{proof}

\begin{define}
The homogenization $\overline{B}_{r,s}^{B} (x,z)$ of $B_{r,s}^{B} (x)$ is defined by
\[ \overline{B}_{r,s}^{B} (x,z) := z^{r+2s} B_{r,s}^{B} (x/z). \]
\end{define}

Let $1\leq j\leq \ell$. Define 
\[ I_{1}^{(j)}= \{ x_{1}, \ldots ,x_{j-1} \},
\,\,\,
I_{2}^{(j)}= \{ x_{j} \},
\,\,\,
I_{3}^{(j)}= \{ x_{j+1}, \ldots , x_{\ell} \}
\] 
Let $\sigma_{k}(y_{1}, y_{2}, \dots)\ \
(k \in \mathbb{Z}_{\ge 0})$ denote the elementary 
symmetric polynomials in $y_{1}, y_{2}, \dots$ 
 of degree $k$.  Then define
\[
\sigma_{k}^{(2,j)} := \sigma_{k}(x_{j}),\  
\tau_{k}^{(3,j)} := \sigma_{k} (x_{j+1}^{2}, \ldots ,x_{\ell}^{2}). 
\]

\begin{define}
\label{basisdefinition} 
Let
$\partial_{i} \,\,
(1\leq i\leq \ell)
$ and $\partial_{z} $ denote 
$\partial/\partial x_{i} $ 
and
$\partial/\partial z $
respectively. 
Define the Euler derivation
\[ \theta_{E} := z \partial_{z} + \sum_{i=1}^{\ell} x_{i} \partial_{i} \] 
and
the following homogeneous derivations 
\begin{multline*}
\varphi_{j}^{B}
:=
(-1)^{j} 
\sum_{i=1}^{\ell} 
\biggl\{ 
\sum_{\substack{N_{1},N_{2} \subset I_{1}^{(j)} \\ N_{1} \cap N_{2} = \emptyset}}
\biggl(
\prod_{x_{t} \in N_{1}} x_{t}^{2}
\biggr)
\biggl(
\prod_{x_{t} \in N_{2}} (-x_{t}z)
\biggr)
\\
\sum_{ \substack{0 \leq k_{2} \leq 1 \\ 0 \leq k_{3} \leq \ell -j }}
(-1)^{k_{2}+k_{3}}
\sigma_{k_{2}}^{(2,j)} \tau_{k_{3}}^{(3,j)}
\ 
\overline{B}_{r, s}^{B} (x_{i},z) 
\biggr\} \partial_{i},
\end{multline*}
where
\[ r := 2\ell -2j-k_{2}-2k_{3}+2\ge 1,
\ \ \
s := |I_{1}^{(j)} \setminus (N_{1} \cup N_{2})|
= (j-1)-|N_{1}|-|N_{2}|\ge 0 \]
for
$1\leq j \leq \ell$.
\end{define}

It is easy to see that each $\varphi_{j}^{B}$ is 
homogeneous derivation of degree $2\ell$ which is equal to 
the
Coxeter number for $B_{\ell} $.  
We will prove
that  the derivations
$\theta_{E}$
and
$\varphi_{1}^{B}, \dots , \varphi_{\ell}^{B}$
form a basis for $D(\Shi (B_{\ell}))$. 
First we will verify the following

\begin{prop}
\label{modeqB}
Let $\varepsilon \in \{ -1,0,1 \}$.
Then we have the following congruence relations:
\[
\overline{B}_{r,s}^{B}(x_{p}, z) + \varepsilon 
\overline{B}_{r,s}^{B}(x_{q}, z)
\equiv 0
\mod (x_{p}+\varepsilon x_{q}),
\]
\begin{multline*}
\overline{B}_{r,s}^{B}(x_{p}, z) + \varepsilon \overline{B}_{r,s}^{B}(x_{q}, z)
\equiv
(x_{p}+ \varepsilon x_{q}) \frac{x_{p}^{r}- (\varepsilon x_{q})^{r}}{x_{p}- \varepsilon x_{q}} (x_{p} \cdot \varepsilon x_{q})^{s}
\mod (x_{p}+\varepsilon x_{q}-z).
\end{multline*} 
\end{prop}

\begin{proof}
The first congruence follows from Definition
\ref{defBofB} {(ii)} and Lemma \ref{oddfunctionB}.
Let the congruent notation $\equiv$ in the following 
calculation be modulo the ideal $(x_{p}+\varepsilon x_{q}-z)$.
By Definition
\ref{defBofB} and Lemma \ref{oddfunctionB}, we have
\begin{align*}
& \overline{B}_{r,s}^{B}(x_{p}, z) + \varepsilon \overline{B}_{r,s}^{B}(x_{q}, z)
= \overline{B}_{r,s}^{B}(x_{p}, z) + \overline{B}_{r,s}^{B}( \varepsilon x_{q}, z) 
\\
&=z^{r+2s} 
\{ B_{r,s}^{B} \left(\frac{x_{p}}{z}\right) + B_{r,s}^{B} 
\left(\frac{\varepsilon x_{q}}{z}\right)\}
\\
&\equiv
(x_{p}+\varepsilon x_{q})^{r+2s} \left\{ B_{r,s}^{B}
\left(\frac{x_{p}}{x_{p}+ \varepsilon  x_{q}}
\right)
+B_{r,s}^{B}
\left(
\frac{ \varepsilon x_{q}}{x_{p}+ \varepsilon x_{q}}
\right)
 \right\} \\
&= 
(x_{p}+\varepsilon x_{q})^{r+2s} \left\{ B_{r,s}^{B}
\left(\frac{x_{p}}{x_{p}+ \varepsilon x_{q}}\right)
-B_{r,s}^{B}\left(
-\frac{ \varepsilon x_{q}}{x_{p}+ \varepsilon x_{q}}\right)
 \right\} \\
&=
(x_{p}+ \varepsilon x_{q})^{r+2s} \frac{\left( \frac{x_{p}}{x_{p}+ \varepsilon x_{q}} \right)^{r} - \left( \frac{ \varepsilon x_{q}}{x_{p}+ \varepsilon x_{q}} \right)^{r}}{\frac{x_{p}}{x_{p}+ \varepsilon x_{q}} - \frac{ \varepsilon x_{q}}{x_{p}+ \varepsilon x_{q}}}
\left( \frac{x_{p}}{x_{p}+ \varepsilon x_{q}} \right)^{s}
\left( \frac{ \varepsilon x_{q}}{x_{p}+ \varepsilon x_{q}} \right)^{s} \\
&= (x_{p}+ \varepsilon x_{q}) \frac{x_{p}^{r}- (\varepsilon x_{q})^{r}}{x_{p}- \varepsilon x_{q}} (x_{p} \cdot \varepsilon x_{q})^{s}.   
\end{align*}
\end{proof}

\begin{prop}
\label{belongB}
The derivations 
$\varphi_{j}^{B}
\,\,
(1\leq j\leq \ell)
$ 
belong to the module
$D(\Shi (B_{\ell}) )$. 
\end{prop}

\begin{proof}
By Proposition
\ref{modeqB}, we first have
\begin{align*}  
\varphi_{j}^{B} (x_{p} + \varepsilon x_{q}) 
&= 
(-1)^{j} 
\sum_{\substack{N_{1},N_{2} \subset I_{1}^{(j)} \\ N_{1} \cap N_{2} = \emptyset}} 
\biggl(
\prod_{x_{t} \in N_{1}} x_{t}^{2}
\biggr)
\biggl(
\prod_{x_{t} \in N_{2}} (-x_{t}z)
\biggr)
\\
&~~~~~~\sum_{\substack{0 \leq k_{2} \leq 1 \\ 0 \leq k_{3} \leq \ell -j}}
(-1)^{k_{2}+k_{3}}
\sigma_{k_{2}}^{(2,j)}
\tau_{k_{3}}^{(3,j)}
(\overline{B}_{r,s}^{B}(x_{p},z) + \varepsilon \overline{B}_{r,s}^{B}(x_{q},z))
\\
&\equiv 0 \mod (x_{p}+\varepsilon x_{q})
\end{align*}
for $1\leq j\leq \ell$.
Thus we conclude 
that $\varphi_{j}^{B}(x_{p}),\varphi_{j}^{B}(x_{p} \pm x_{q})$
are divisible by $x_{p},x_{p} \pm x_{q}$ for $1\leq p \leq \ell,1\leq p<q \leq \ell$ respectively.

Let the congruent notation $\equiv$ in the following 
calculation be modulo the ideal $(x_{p}+\varepsilon x_{q}-z)$.
By Proposition
\ref{modeqB}, for $1\leq j\leq \ell$, we also have 
\begin{align*}  
& \hspace{3ex} \varphi_{j}^{B} (x_{p}+\varepsilon x_{q} -z) 
= \varphi_{j}^{B} (x_{p}+\varepsilon x_{q}) \\
&= (-1)^{j} 
\sum_{\substack{N_{1},N_{2} \subset I_{1}^{(j)} \\ N_{1} \cap N_{2} = \emptyset}} 
\biggl(
\prod_{x_{t} \in N_{1}} x_{t}^{2}
\biggr)
\biggl(
\prod_{x_{t} \in N_{2}} (-x_{t}z)
\biggr)
\\
& ~~~~~
\sum_{\substack{0 \leq k_{2} \leq 1 \\ 0 \leq k_{3} \leq \ell -j}}
(-1)^{k_{2}+k_{3}}
\sigma_{k_{2}}^{(2,j)}
\tau_{k_{3}}^{(3,j)}
(\overline{B}_{r,s}^{B}(x_{p},z) + \varepsilon \overline{B}_{r,s}^{B}(x_{q},z))
\\
&\equiv 
(-1)^{j} 
\sum_{\substack{N_{1},N_{2} \subset I_{1}^{(j)} \\ N_{1} \cap N_{2} = \emptyset}} 
\biggl(
\prod_{x_{t} \in N_{1}} x_{t}^{2}
\biggr)
\biggl(
\prod_{x_{t} \in N_{2}} (-x_{t}(x_{p}+\varepsilon x_{q}))
\biggr)
\\
& \hspace{3em}
\sum_{\substack{0 \leq k_{2} \leq 1 \\ 0 \leq k_{3} \leq \ell -j}}
(-1)^{k_{2}+k_{3}}
\sigma_{k_{2}}^{(2,j)}
\tau_{k_{3}}^{(3,j)}
(x_{p}+\varepsilon x_{q})
\frac{x_{p}^{r}-(\varepsilon x_{q})^{r}}{x_{p}-\varepsilon x_{q}}
(x_{p} \cdot \varepsilon x_{q})^{s}
\\
&= (x_{p}+\varepsilon x_{q})
\sum_{\substack{N_{1},N_{2} \subset I_{1}^{(j)} \\ N_{1} \cap N_{2} = \emptyset}} 
\biggl(
\prod_{x_{t} \in N_{1}} x_{t}^{2}
\biggr)
\biggl(
\prod_{x_{t} \in N_{2}} (-x_{t}(x_{p}+\varepsilon x_{q}))
\biggr)
(x_{p} \cdot \varepsilon x_{q})^{s}
\\
& \hspace{5em}
\frac{(-1)^{\ell+1}}{x_{p}-\varepsilon x_{q}}
\biggl\{
\sum_{\substack{0 \leq k_{2} \leq 1 \\ 0 \leq k_{3} \leq \ell -j}}
(-1)^{\ell -j+1-k_{2}-k_{3}} 
\sigma_{k_{2}}^{(2,j)}
\tau_{k_{3}}^{(3,j)}
x_{p}^{r} 
\\
&\hspace{8em}
-\sum_{\substack{0 \leq k_{2} \leq 1 \\ 0 \leq k_{3} \leq \ell -j}}
(-1)^{\ell -j+1-k_{2}-k_{3}} 
\sigma_{k_{2}}^{(2,j)}
\tau_{k_{3}}^{(3,j)}
(\varepsilon x_{q})^{r}
\biggr\}.
\end{align*}
Here,
\begin{align*}
& \sum_{\substack{N_{1},N_{2} \subset I_{1}^{(j)} \\ N_{1} \cap N_{2} = \emptyset}} 
\biggl(
\prod_{x_{t} \in N_{1}} x_{t}^{2}
\biggr)
\biggl(
\prod_{x_{t} \in N_{2}} (-x_{t}(x_{p}+\varepsilon x_{q}))
\biggr)
(x_{p} \cdot \varepsilon x_{q})^{s}
\\
&= \prod_{t=1}^{j-1} (x_{t}^2-(x_{p}+\varepsilon x_{q})x_{t}+x_{p} \cdot \varepsilon x_{q})
= \prod_{t=1}^{j-1} (x_{t}-x_{p})(x_{t}-\varepsilon x_{q})
\end{align*}
and
\begin{align*}
& \sum_{\substack{0 \leq k_{2} \leq 1 \\ 0 \leq k_{3} \leq \ell -j}}
(-1)^{\ell -j+1-k_{2}-k_{3}} 
\sigma_{k_{2}}^{(2,j)}
\tau_{k_{3}}^{(3,j)}
x_{p}^{r} 
\\
&= x_{p} \sum_{k_{2}=0}^{1}
\sigma_{k_{2}}^{(2,j)} (-x_{p})^{1-k_{2}}
\sum_{k_{3}=0}^{\ell -j}
\tau_{k_{3}}^{(3,j)} (-x_{p}^2)^{\ell -j-k_{3}}
= x_{p} (x_{j}-x_{p}) \prod_{t=j+1}^{\ell} (x_{t}^{2}-x_{p}^{2}).
\end{align*}
If $1 \leq p \leq j-1$, then
\[
\prod_{t=1}^{j-1} (x_{t}-x_{p})(x_{t}-\varepsilon x_{q})
= 0.
\]
If $j \leq p < q \leq \ell$, then
\[
x_{p} (x_{j}-x_{p}) 
\biggl(
\prod_{t=j+1}^{\ell} (x_{t}^{2}-x_{p}^{2})
\biggr) 
=
\varepsilon x_{q} (x_{j}-\varepsilon x_{q}) 
\biggl(
\prod_{t=j+1}^{\ell} (x_{t}^{2}-(\varepsilon x_{q})^{2})
\biggr)
= 0.
\]
Therefore
\begin{align*}
& \hspace{3ex} 
\varphi_{j}^{B}(x_{p}+\varepsilon x_{q}-z)
\\
& \equiv
(-1)^{\ell +1}
\frac{x_{p}+\varepsilon x_{q}}{x_{p}-\varepsilon x_{q}}
\prod_{t=1}^{j-1} (x_{t}-x_{p})(x_{t}-\varepsilon x_{q})
\\
& \hspace{3em}
\biggl\{
x_{p} (x_{j}-x_{p}) 
\biggl(
\prod_{t=j+1}^{\ell} (x_{t}^{2}-x_{p}^{2})
\biggr) 
-\varepsilon x_{q} (x_{j}-\varepsilon x_{q}) 
\biggl(
\prod_{t=j+1}^{\ell} (x_{t}^{2}-(\varepsilon x_{q})^{2})
\biggr)
\biggr\}
\\
& =0
\end{align*}
for all pairs $(p,q)$ with $1\leq p < q \leq \ell $ 
and $\varepsilon \in \{ -1,0,1 \}$. 
Hence $\varphi_{j}^{B} \in D(\mathcal{S}(B_{\ell}))$ for $1\leq j\leq \ell$.
\end{proof}

\section{The $W$-equivariance}
Recall that
$\A(\Phi)$ is the Weyl arrangement in $E$ corresponding
to the irreducible
root system $\Phi.$
Then we may identify
$$
\overline{S} := S/zS
\simeq{\mathbb R}[x_{1} , \dots , x_{\ell}]$$
with the algebra of polynomial functions on $E$.
In
\cite{ST98}
L. Solomon and H. Terao studied 
the $\overline{S}$-module
\[ 
D(\mathcal{A}(\Phi), 2) 
:= \{ \theta \in \mathrm{Der}(\overline{S}) \mid 
\theta (\alpha_{H}) \in \overline{S} \alpha_{H}^{2}, H \in \mathcal{A}(\Phi)
 \},
\]
which was denoted by $E(\A)$ in 
\cite{ST98}.
Let $h$ be the Coxeter number for $\Phi$.  Define
\[ 
D(\mathcal{A}(\Phi), 2)_{h}  
:= \{ \theta \in 
D(\mathcal{A}(\Phi), 2)
\mid 
\deg \theta=h \}
\cup \{ 0 \},
\]
which is a real
vector space.  Note that the Weyl group $W$ corresponding to $\Phi$
naturally acts on
$D(\mathcal{A}(\Phi), 2)$ and
$D(\mathcal{A}(\Phi), 2)_{h} $.
We also define an $S$-submodule 
\[
D_{0} (\Shi(\Phi)):= \{\varphi\in D(\Shi(\Phi)) \mid \varphi(z)=0\}
\]
of 
$D (\Shi(\Phi))$.
 Then $D(\Shi(\Phi))$ has a decomposition
\[
D(\Shi(\Phi))=
D_{0} (\Shi(\Phi))
\oplus
S \theta_{E}
\]
over $S$. 
Let
\[
D_{0} (\Shi(\Phi))_{h} := 
\{\varphi\in D_{0} (\Shi(\Phi)) \mid \deg \varphi = h\}
\cup \{ 0 \},
\]
which is a real vector space.
If $\varphi\in D_{0} ({\mathcal S}(\Phi))$, 
then
$\varphi(\alpha_{H}) \in \alpha_{H}(\alpha_{H}-z) S$
 for any $H \in \mathcal{A}(\Phi)$.
Let $\overline{\varphi} := \varphi|_{z=0}$
be the restriction of $\varphi$ to $z=0$.
Then
$\overline{\varphi}(\alpha_{H}) \in \alpha_{H}^{2} \overline{S}$ 
for any $H \in 
\mathcal{A}(\Phi)$,
hence $\overline{\varphi} \in D(\mathcal{A}(\Phi), 2)$.  

\begin{theorem}
\label{Wequivariance}

(1) (L. Solomon-H. Terao\cite{ST98})
The $\overline{S} $-module $D(\mathcal{A}(\Phi), 2)$ 
 is a free module with a basis consisting of
$\ell$  derivations homogeneous of degree $h$.
In other words, we have an isomorphism
\[
D(\mathcal{A}(\Phi), 2)
\simeq 
D(\mathcal{A}(\Phi), 2)_{h} \otimes_{\mathbb R} \overline{S}.
\]

(2)
(M. Yoshinaga\cite{Y04})
 The $S$-module  $D_{0} (\Shi(\Phi))$ 
is a free 
module with a basis consisting of
$\ell$  derivations homogeneous of degree $h$.
In other words, we have an isomorphism
\[
D_{0}(\Shi(\Phi))\simeq D(\Shi(\Phi))_{h} \otimes_{\mathbb R} S.
\]
Also the restriction map
\[
\rho:
D_{0} (\Shi(\Phi))_{h} \longrightarrow 
D(\mathcal{A}(\Phi), 2)_{h} 
\]
 defined by $\varphi \mapsto \overline{\varphi}={\varphi|_{z=0} }$ 
is a linear isomorphism.
\end{theorem}

Suppose that $\Phi$ is of the type $B_{\ell}$
in the rest of this section.
Then   we may define an explicit  
$\mathbb{R}$-linear map 
$$
\Psi : E^{*} \rightarrow D_{0} (\Shi(B_{\ell} ))_{h} 
$$
 by
$$\Psi (x_{j}) = {\varphi}_{j}^{B}
\,\,\,\,
(1\le j\le \ell)$$
using the derivations
$
{\varphi}_{1}^{B}, 
\dots,  
{\varphi}_{\ell}^{B}
$ in Definition \ref{basisdefinition}.

%
%
%
%
%
%
%

\begin{theorem}
\label{unique}
Let $\Phi$ be a root system of the type $B_{\ell}$. 

(1)
The map $$
\Xi : E^{*} \rightarrow D(\mathcal{A}(B_{\ell} ), 2)_{h}
$$
defined by $\Xi  = \rho\circ\Psi$ is a $W$-equivariant isomorphism.
%

(2)
The map 
$$\Psi : E^{*} \rightarrow D_{0} (\Shi(B_{\ell} ))_{h} 
$$ is a linear isomorphism. 
\end{theorem}

\begin{proof}
%

(1)
Since
\[
\overline{B}_{r,s}^{B} (x_{i},0) 
=
\begin{cases}
(-1)^{s}x_{i}^{r+2s} / (r+2s)& (r:\text{odd number}) \\
0 & (r:\text{even number})
\end{cases},
\]
\begin{align*}
\Xi(x_{j})(x_{i}) &= (\rho\circ \Psi(x_{j}))(x_{i})
=
\varphi_{j}^{B}  (x_{i})|_{z=0} 
\\
&= 
(-1)^{j} 
x_{j}
\sum_{N_{1} \subset I_{1}^{(j)}}
\biggl(
\prod_{x_{t} \in N_{1}} x_{t}^{2}
\biggr)
\sum_{k_{3}=0}^{\ell -j}
(-1)^{1+k_{3}}
\tau_{k_{3}}^{(3,j)}
(-1)^{s}
\frac{x_{i}^{r+2s}}{r+2s}
\\
&= 
(-1)^{j} 
x_{j}
\sum_{m=0}^{j-1}
\sum_{\substack{N_{1} \subset I_{1}^{(j)} \\ |N_{1}|=m}}
\biggl(
\prod_{x_{t} \in N_{1}} x_{t}^{2}
\biggr)
\sum_{k_{3}=0}^{\ell -j}
(-1)^{1+k_{3}}
\tau_{k_{3}}^{(3,j)}
(-1)^{s}
\frac{x_{i}^{r+2s}}{r+2s}
\\
&= x_{j}
\sum_{m=0}^{j-1}
\tau_{m}^{(1,j)}
\sum_{k_{3}=0}^{\ell -j}
(-1)^{m+k_{3}}
\tau_{k_{3}}^{(3,j)}
\frac{x_{i}^{2\ell -2m-2k_{3}-1}}{2\ell -2m-2k_{3}-1}
\\
&= x_{j}
\sum_{k=0}^{\ell -1}
(-1)^{k}
\sigma_{k}(x_{1}^{2},\ldots ,x_{j-1}^{2},x_{j+1}^{2},\ldots ,x_{\ell}^{2})
\frac{x_{i}^{2\ell -2k-1}}{2\ell -2k-1}.
\end{align*}
Thus we obtain
\begin{align*}
\Xi (x_{j})
= x_{j}
\sum_{k=0}^{\ell -1}
(-1)^k
\sigma_{k}(x_{1}^{2},\ldots ,x_{j-1}^{2},x_{j+1}^{2},\ldots ,x_{\ell}^{2})
\sum_{i=1}^{\ell} 
 \left(\frac{x_{i}^{2\ell -2k -1}}{2\ell -2k -1}\right) \partial_{i}. 
\end{align*}
Since
$$\sum_{i=1}^{\ell} 
 \left(\frac{x_{i}^{2\ell -2k -1}}{2\ell -2k -1}\right) \partial_{i}$$
is a $W$-invariant derivation
and the correspondence  
\[
x_{j} \mapsto x_{j}\sigma_{k}(x_{1}^{2},\ldots ,x_{j-1}^{2},x_{j+1}^{2},\ldots ,x_{\ell}^{2}) 
\,\,\,\,\,\,
(0\le k\le \ell-1)
\]
is $W$-equivariant for every $k\in{\mathbb Z}_{\ge 0} $, 
we conclude that
$\Xi$ is  $W$-equivariant.
Therefore
$\Xi$ is bijective
by Schur's lemma.

(2) follows from (1) because the restriction map $\rho$ is bijective by Theorem \ref{Wequivariance} (2).
\end{proof}

\begin{theorem}
\label{mainB} 
The derivations $\theta_{E}, \varphi_{1}^{B}, \dots, \varphi_{\ell}^{B}$ 
form a basis for $D(\Shi(B_{\ell}))$. 
\end{theorem}

\begin{proof}
It is enough to show that
$\varphi_{1}^{B}, \dots, \varphi_{\ell}^{B}$ 
form a basis for $D_{0} (\Shi(B_{\ell}))$.
Recall that each $\Psi(x_{j})
=\varphi_{j}^{B}$ belongs to 
$D_{0} (\Shi(B_{\ell}))_{h} $.
Theorems \ref{Wequivariance} (2) and \ref{unique} (2) complete the proof.
\end{proof}

\section{A basis construction for the type $C_{\ell}$}

\begin{define} 
For $(r,s) \in \mathbb{Z}_{> 0}\times \mathbb{Z}_{\ge 0}$, 
define a polynomial
$B_{r,s}^{C} (x)$ in $x$ satisfying the following two conditions:\\

(i) 
$B_{r,s}^{C} (x+1)-B_{r,s}^{C} (x)
=\{ (x+1)^{r-1}+(-x)^{r-1} \} (x+1)^{s} (-x)^{s}$,\\

(ii) $B_{r,s}^{C} (0)=0.$
\end{define}

It is easy to see that 
$B_{r,s}^{C} (x)$
uniquely exists and
$$
\deg B_{r,s}^{C} (x)
=
\begin{cases}
r+2s   &{\rm ~if~} r {\rm ~is~odd},\\
r+2s-1 &{\rm ~if~} r {\rm ~is~even}.
\end{cases}  
$$
The following lemma can be proved by a smilar argument to
the proof of Lemma \ref{oddfunctionB}: 

\begin{lemma}
\label{oddfunctionC}
$B_{r,s}^{C} (x)$ is an odd function.
\end{lemma}

\begin{define}
The homogenization $\overline{B}_{r,s}^{C} (x,z)$ of $B_{r,s}^{C} (x)$ is defined by
\[ \overline{B}_{r,s}^{C} (x,z) := z^{r+2s} B_{r,s}^{C} (x/z). \]
\end{define}

\begin{define}
\label{basisdefinitionC} 
Define homogeneous derivations 
\begin{multline*}
\varphi_{j}^{C}
:=
(-1)^{j} 
\sum_{i=1}^{\ell }
\biggl\{ 
\sum_{\substack{N_{1},N_{2} \subset I_{1}^{(j)} \\ N_{1} \cap N_{2} = \emptyset}}
\biggl(
\prod_{x_{t} \in N_{1}} x_{t}^{2}
\biggr)
\biggl(
\prod_{x_{t} \in N_{2}} (-x_{t}z)
\biggr)
\\
\sum_{ \substack{0 \leq k_{2} \leq 1 \\ 0 \leq k_{3} \leq \ell -j }}
(-1)^{k_{2}+k_{3}}
\sigma_{k_{2}}^{(2,j)} \tau_{k_{3}}^{(3,j)}
\ 
\overline{B}_{r, s}^{C} (x_{i},z) 
\biggr\} \partial_{i}
\end{multline*}
where
\[ r := 2\ell -2j-k_{2}-2k_{3}+2\ge 1, \,\,\,
s := |I_{1}^{(j)} \setminus (N_{1} \cup N_{2})|
= (j-1)-|N_{1}|-|N_{2}|\ge 0
\]
for
$1\leq j \leq \ell$.
\end{define}

Note that 
$\varphi_{j}^{C}$
 is exactly the same as
$\varphi_{j}^{B}$
with only one exception:
the use of 
$B_{r,s}^{C}(x_{i},z)$
instead of
$B_{r,s}^{B}(x_{i},z)$.
Thus each $\varphi_{j}^{B}$ is 
homogeneous derivation of degree $2\ell$ which is equal to 
the
Coxeter number for $C_{\ell}$.  
We will prove
that  the derivations
$\theta_{E}$
and
$\varphi_{1}^{C}, \dots , \varphi_{\ell}^{C}$
form a basis for $D(\Shi (C_{\ell}))$. 
We first have the following Proposition: 

\begin{prop}
\label{modeqC}
Let $\varepsilon \in \{ -1,0,1 \}$.
Then we have 
the following conguruence relations:
\[
\overline{B}_{r,s}^{C}(x_{p}, z) + \varepsilon \overline{B}_{r,s}^{C}(x_{q}, z)
\equiv 0
\mod (x_{p}+\varepsilon x_{q}),
\]
\begin{multline*} 
\overline{B}_{r,s}^{C}(x_{p}, z) + 
\varepsilon \overline{B}_{r,s}^{C}(x_{q}, z)
\equiv
(x_{p}+ \varepsilon x_{q}) 
\{ x_{p}^{r-1} + (\varepsilon x_{q})^{r-1} \} 
(x_{p} \cdot \varepsilon x_{q})^{s}\\ 
\mod (x_{p}+\varepsilon x_{q}-z).
\end{multline*} 
\end{prop}

\begin{proof}
Imitate the proof of Proposition \ref{modeqB}.
\end{proof}

\begin{prop}
The derivations 
$\varphi_{j}^{C}
\,\,
(1\leq j\leq \ell)
$ 
belong to the module
$D(\Shi (C_{\ell}) )$. 
\end{prop}

\begin{proof}
By Proposition
\ref{modeqC}, we first have
\begin{align*}  
& \varphi_{j}^{C} (x_{p} + \varepsilon x_{q}) 
\\
&= (-1)^{j} \sum_{\substack{N_{1},N_{2} \subset I_{1}^{(j)} \\ N_{1} \cap N_{2} = \emptyset}} 
\biggl(
\prod_{x_{t} \in N_{1}} x_{t}^{2}
\biggr)
\biggl(
\prod_{x_{t} \in N_{2}} (-x_{t}z)
\biggr)
\\
& \hspace{7em}
\sum_{\substack{0 \leq k_{2} \leq 1 \\ 0 \leq k_{3} \leq \ell -j}}
(-1)^{k_{2}+k_{3}}
\sigma_{k_{2}}^{(2,j)}
\tau_{k_{3}}^{(3,j)}
(\overline{B}_{r,s}^{C}(x_{p},z) + \varepsilon \overline{B}_{r,s}^{C}(x_{q},z))
\\
&\equiv 0 \ \ \ \ \ (\mathrm{mod}\ (x_{p}+\varepsilon x_{q}))
\end{align*}
for $1\leq j\leq \ell$.
Thus we 
conclude that $\varphi_{j}^{C}(2x_{p}),\varphi_{j}^{C}(x_{p} \pm x_{q})$
are divisible by $2x_{p},x_{p} \pm x_{q}$ for $1\leq p \leq \ell,1\leq p<q \leq \ell$ respectively.

Let the congruent notation $\equiv$ in the following 
calculation be modulo the ideal $(x_{p}+\varepsilon x_{q}-z)$.
By Proposition \ref{modeqC}, for $1\leq j\leq \ell$, we also have
\begin{align*}  
& \hspace{3ex} \varphi_{j}^{C} (x_{p}+\varepsilon x_{q} -z) 
=  
\varphi_{j}^{C} (x_{p}+\varepsilon x_{q}) \\
&=(-1)^{j} \sum_{\substack{N_{1},N_{2} \subset I_{1}^{(j)} \\ N_{1} \cap N_{2} = \emptyset}} 
\biggl(
\prod_{x_{t} \in N_{1}} x_{t}^{2}
\biggr)
\biggl(
\prod_{x_{t} \in N_{2}} (-x_{t}z)
\biggr)\\
& \hspace{7em}
\sum_{\substack{0 \leq k_{2} \leq 1 \\ 0 \leq k_{3} \leq \ell -j}}
(-1)^{k_{2}+k_{3}}
\sigma_{k_{2}}^{(2,j)}
\tau_{k_{3}}^{(3,j)}
(\overline{B}_{r,s}^{C}(x_{p},z) + \varepsilon \overline{B}_{r,s}^{C}(x_{q},z))
\end{align*}  
\begin{align*}  
&\equiv 
(-1)^{j} 
\sum_{\substack{N_{1},N_{2} \subset I_{1}^{(j)} \\ N_{1} \cap N_{2} = \emptyset}} 
\biggl(
\prod_{x_{t} \in N_{1}} x_{t}^{2}
\biggr)
\biggl(
\prod_{x_{t} \in N_{2}} (-x_{t}(x_{p}+\varepsilon x_{q}))
\biggr)
\\
& \hspace{4em}
\sum_{\substack{0 \leq k_{2} \leq 1 \\ 0 \leq k_{3} \leq \ell -j}}
(-1)^{k_{2}+k_{3}}
\sigma_{k_{2}}^{(2,j)}
\tau_{k_{3}}^{(3,j)}
(x_{p}+\varepsilon x_{q})
\{ 
x_{p}^{r-1}+(\varepsilon x_{q})^{r-1}
\}
(x_{p} \cdot \varepsilon x_{q})^{s}
\\
&= (x_{p}+\varepsilon x_{q})
\sum_{\substack{N_{1},N_{2} \subset I_{1}^{(j)} \\ N_{1} \cap N_{2} = \emptyset}} 
\biggl(
\prod_{x_{t} \in N_{1}} x_{t}^{2}
\biggr)
\biggl(
\prod_{x_{t} \in N_{2}} (-x_{t}(x_{p}+\varepsilon x_{q}))
\biggr)
(x_{p} \cdot \varepsilon x_{q})^{s}
\\
& \hspace{4em}
(-1)^{\ell+1}
\biggl\{
\sum_{\substack{0 \leq k_{2} \leq 1 \\ 0 \leq k_{3} \leq \ell -j}}
(-1)^{\ell -j+1-k_{2}-k_{3}} 
\sigma_{k_{2}}^{(2,j)}
\tau_{k_{3}}^{(3,j)}
x_{p}^{r-1} 
\\
&\hspace{8em}
+\sum_{\substack{0 \leq k_{2} \leq 1 \\ 0 \leq k_{3} \leq \ell -j}}
(-1)^{\ell -j+1-k_{2}-k_{3}} 
\sigma_{k_{2}}^{(2,j)}
\tau_{k_{3}}^{(3,j)}
(\varepsilon x_{q})^{r-1}
\biggr\}.
\end{align*}
Here,
\begin{align*}
& \sum_{\substack{N_{1},N_{2} \subset I_{1}^{(j)} \\ N_{1} \cap N_{2} = \emptyset}} 
\biggl(
\prod_{x_{t} \in N_{1}} x_{t}^{2}
\biggr)
\biggl(
\prod_{x_{t} \in N_{2}} (-x_{t}(x_{p}+\varepsilon x_{q}))
\biggr)
(x_{p} \cdot \varepsilon x_{q})^{s}
\\
&= \prod_{t=1}^{j-1} (x_{t}^2-(x_{p}+\varepsilon x_{q})x_{t}+x_{p} \cdot \varepsilon x_{q})
= \prod_{t=1}^{j-1} (x_{t}-x_{p})(x_{t}-\varepsilon x_{q}), 
\end{align*}
and
\begin{align*}
& \sum_{\substack{0 \leq k_{2} \leq 1 \\ 0 \leq k_{3} \leq \ell -j}}
(-1)^{\ell -j+1-k_{2}-k_{3}} 
\sigma_{k_{2}}^{(2,j)}
\tau_{k_{3}}^{(3,j)}
x_{p}^{r-1} 
\\
&= \sum_{k_{2}=0}^{1}
\sigma_{k_{2}}^{(2,j)} (-x_{p})^{1-k_{2}}
\sum_{k_{3}=0}^{\ell -j}
\tau_{k_{3}}^{(3,j)} (-x_{p}^2)^{\ell -j-k_{3}}
= (x_{j}-x_{p}) \prod_{t=j+1}^{\ell} (x_{t}^{2}-x_{p}^{2}).
\end{align*}
If $1 \leq p \leq j-1$, then
\[
\prod_{t=1}^{j-1} (x_{t}-x_{p})(x_{t}-\varepsilon x_{q})
= 0.
\]
If $j \leq p < q \leq \ell$ and $\varepsilon \in \{ -1,1 \}$, then
\[
(x_{j}-x_{p}) 
\prod_{t=j+1}^{\ell} (x_{t}^{2}-x_{p}^{2})
=
(x_{j}-\varepsilon x_{q}) 
\prod_{t=j+1}^{\ell} (x_{t}^{2}-(\varepsilon x_{q})^{2})
= 0.
\]
Therefore
\begin{align*}
& \hspace{3ex} 
\varphi_{j}^{C}(x_{p}+\varepsilon x_{q}-z)
\\
& \equiv
(-1)^{\ell -j+1}
(x_{p}+\varepsilon x_{q})
\prod_{t=1}^{j-1} (x_{t}-x_{p})(x_{t}-\varepsilon x_{q})
\\
& \hspace{3em}
\biggl\{
(x_{j}-x_{p}) 
\biggl(
\prod_{t=j+1}^{\ell} (x_{t}^{2}-x_{p}^{2})
\biggr) 
+(x_{j}-\varepsilon x_{q}) 
\biggl(
\prod_{t=j+1}^{\ell} (x_{t}^{2}-(\varepsilon x_{q})^{2})
\biggr)
\biggr\}
\\
&=0
\end{align*}
for all pairs $(p,q)$ with $1\leq p < q \leq \ell $ 
where $\varepsilon \in \{ -1,1 \}$.
When $p=q, \varepsilon =1$,
\begin{align*}
&\varphi_{j}^{C} (x_{p}+\varepsilon x_{q}-z)
= \varphi_{j}^{C} (2x_{p}-z)
\\
&\equiv
(-1)^{\ell -j+1}
(2x_{p})
\prod_{t=1}^{j-1} (x_{t}-x_{p})^{2}
\biggl\{
2(x_{j}-x_{p}) 
\prod_{t=j+1}^{\ell} (x_{t}^{2}-x_{p}^{2})
\biggr\}
\\
&=0
\end{align*}
for $1\leq p \leq \ell$.
Hence $\varphi_{j} \in D(\mathcal{S}(C_{\ell} ))$ for $1\leq j\leq \ell$.
\end{proof}

We may define an explicit  
$\mathbb{R}$-linear map 
$$
\Psi : E^{*} \rightarrow D_{0} (\Shi(C_{\ell} ))_{h} 
$$
 by
$$\Psi (x_{j}) = {\varphi}_{j}^{C}
\,\,\,\,
(1\le j\le \ell)$$
using the derivations
$
{\varphi}_{1}^{C}, 
\dots,  
{\varphi}_{\ell}^{C}
$ in Definition \ref{basisdefinitionC}.

\begin{theorem}
\label{uniqueC}
Let $\Phi$ be a root system of the type $C_{\ell}$. 

(1)
The map $$
\Xi : E^{*} \rightarrow D(\mathcal{A}(C_{\ell} ), 2)_{h}
$$
defined by $\Xi  = \rho\circ\Psi$ is a $W$-equivariant isomorphism.
%

(2)
The map 
$$\Psi : E^{*} \rightarrow D_{0} (\Shi(C_{\ell} ))_{h} 
$$ is a linear isomorphism. 
\end{theorem}

\begin{proof}
Since
\[
\overline{B}_{r,s}^{C} (x_{i},0) 
=
2\overline{B}_{r,s}^{B} (x_{i},0) 
=
\begin{cases}
(-1)^{s} 2 x_{i}^{r+2s} / (r+2s)& (r:\text{odd number}) \\
0 & (r:\text{even number})
\end{cases},
\]
we may prove this theorem in the same way as Theorem \ref{unique}.
\end{proof}

\begin{theorem}
\label{mainC} 
The derivations $\theta_{E}, \varphi_{1}^{C}, \dots, \varphi_{\ell}^{C}$ 
form a basis for $D(\Shi(C_{\ell}))$. 
\end{theorem}

\begin{proof}
Apply Theorems \ref{uniqueC} (2)
and and \ref{Wequivariance} (2) in the same way as the proof of
Theorem \ref{mainB}.
\end{proof}

\begin{rem}
Since the $W$-equivariant
isomorphism $\Xi: E^{*} \rightarrow D(\A(B_{\ell}), 2)_{h} $ 
in Theorem \ref{unique} (1)
is unique up to a nonzero constant multiple by Schur's lemma,
the derivations $
\varphi_{1}^{B}|_{z=0},
\dots,
\varphi_{\ell}^{B}|_{z=0}$ 
coincide with the Solomon-Terao basis in \cite{ST98}
up to a nonzero constant multiple.
Therefore,
  our construction of
$
\varphi_{1}^{B},
\dots,
\varphi_{\ell}^{B}
$ 
can be regarded as an explicit realization of
the basis existence theorem by M. Yoshinaga
in \cite{Y04}.
This is also
true for the type $C_{\ell}$.  
\end{rem}

{\bf Acknowledgment.}
The author is deeply grateful to Professor H. Terao
for his advice and support.

 \vspace{5mm}


\begin{thebibliography}{99}

\bibitem{Ath} Ch. Athanasiadis, 
On free deformations of the braid arrangement. 
\textit{European J. Combin.} \textbf{19} (1998), 7-18.

\bibitem{ER} P. H. Edelman and V. Reiner, 
Free arrangements and rhombic tilings. \textit{Discrete Comp. Geom.} 
\textbf{15} (1996), 
307-340.

\bibitem{Gao} R. Gao, D. Pei and H. Terao, 
The Shi arrangement of the type $D_{\ell}$.
\textit{Proc. Japan Acad. Ser. A, Math. Sci.} \textbf{88} (2012), 41-45.

\bibitem{Hea} P. Headley, 
On a family of hyperplane arrangements
related to the affine Weyl groups.
\textit{J. Algebraic Combin.} \textbf{6} (1997), 331-338.

\bibitem{OT} P. Orlik and H. Terao, 
\textit{Arrangements of hyperplanes}.
Grundlehren der Mathematischen Wissenschaften, 
\textbf{300}. Springer-Verlag, Berlin, 1992.

%
%
%

\bibitem{shi1} J.-Y. Shi, 
The Kazhdan-Lusztig cells in
certain affine Weyl groups. 
\textit{Lecture Notes in Math.}, {\bf 1179},
Springer-Verlag, 1986.
 
%

\bibitem{ST98} 
L. Solomon and H. Terao, The double Coxeter arrangements. 
\textit{Comment. Math. Helv.} \textbf{73} (1998), 237--258. 

\bibitem{St96} R. P. Stanley, Hyperplane arrangements, interval orders
and trees. 
\textit{Proc. Natl. Acad. Sci}., \textbf{93} (1996), 2620--2625.

\bibitem{Sute}
D. Suyama, H. Terao, The Shi arrangements and the Bernoulli polynomials. \textit{Bull. London Math. Soc.} (to appear)




\bibitem{Y04} M. Yoshinaga, Characterization of a free arrangement and
conjecture of
Edelman and Reiner. \textit{Invent. Math.} \textbf{157} (2004), 
449--454.



 \end{thebibliography}
\end{document}